\documentclass[12pt]{article}

\RequirePackage{kpfonts}
\RequirePackage[sf,bf,small,raggedright]{titlesec}
\usepackage{textcomp}
\usepackage{amssymb}
\usepackage{bbm}
\usepackage{fancyhdr}
\usepackage{amsmath}
\usepackage{amsbsy,amsthm}
\usepackage{amscd}
\usepackage{latexsym}
\usepackage{graphicx}   
\usepackage{pdfsync}
\usepackage{blkarray}
\usepackage{multirow}
\usepackage{hyperref}
\usepackage{color}
\usepackage{comment}
\usepackage{framed}

\newcommand{\R}{\mathbb{R}}

\newcommand{\ind}{\mathbbm{1}}

\newcommand{\E}{\mathbb{E}}

\newcommand{\eps}{\varepsilon}

\newtheorem{lemma}{Lemma}
\newtheorem{theorem}{Theorem}

\newtheorem{proposition}{Proposition}

\numberwithin{equation}{section}

\def\IND{\mathbbm{1}}

\newcommand{\ol}{\overline}

\newcommand{\EXP}{\mathbb{E}}
\renewcommand{\E}{\mathbb{E}}
\newcommand{\PROB}{\mathbb{P}}

\newcommand{\var}{\mathrm{Var}}

\newcommand{\ER}{Erd\H{o}s-R\'enyi }

\date{}

\begin{document}

\title{Concentration of the spectral norm of Erd\H{o}s-R\'enyi random graphs
\thanks{
G\'abor Lugosi was supported by
the Spanish Ministry of Economy and Competitiveness,
Grant MTM2015-67304-P and FEDER, EU,
by ``High-dimensional problems in structured probabilistic models -
Ayudas Fundaci\'on BBVA a Equipos de Investigaci\'on Cientifica 2017'' 
and by ``Google Focused Award Algorithms and Learning for AI''.
Shahar Mendelson was supported in part by the Israel Science Foundation. 
Nikita Zhivotovskiy was supported by RSF grant No. 18-11-00132.
}
}
\author{
G\'abor Lugosi\thanks{Department of Economics and Business, Pompeu
  Fabra University, Barcelona, Spain, gabor.lugosi@upf.edu}
\thanks{ICREA, Pg. Lluís Companys 23, 08010 Barcelona, Spain}
\thanks{Barcelona Graduate School of Economics}
\and
Shahar Mendelson \thanks{Department of Mathematics, Technion, I.I.T, and Mathematical Sciences Institute, The Australian National University, shahar@tx.technion.ac.il}
\and
Nikita Zhivotovskiy \thanks{Department of Mathematics, Technion, I.I.T, and National University HSE, nikita.zhivotovskiy@phystech.edu}
}

\maketitle

\begin{abstract}
We present results on the concentration properties of the spectral norm $\|A_p\|$ of the adjacency matrix $A_p$ of an Erd\H{o}s-R\'enyi 
random graph $G(n,p)$. First we consider the Erd\H{o}s-R\'enyi random graph \emph{process} and prove that $\|A_p\|$ is 
\emph{uniformly} concentrated over the range $p\in [C\log n/n,1]$.
The analysis is based on delocalization arguments, uniform laws of large numbers, together with the entropy method to prove concentration inequalities. As an application of our techniques we prove sharp sub-Gaussian moment inequalities for $\|A_p\|$ for all $p\in [c\log^3n/n,1]$ that improve the general bounds of Alon, Krivelevich, and Vu \cite{AlKrVu01} and some of the more recent results of Erd{\H{o}}s et al. \cite{ErKnYaYi13}. Both results are consistent with the asymptotic result of F{\"u}redi and Koml{\'o}s \cite{FuKo81} that holds for fixed $p$ as $n\to \infty$.
\end{abstract}

\section{Introduction}

An Erd\H{o}s-R\'enyi random graph $G(n,p)$, named after the authors of the pioneering work \cite{ErRe60},  
is a graph defined on the vertex set $[n]=\{1,\ldots,n\}$  in which any two vertices $i,j \in [n]$, $i\neq j$, are
connected by an edge independently, with probability $p$. Such a random graph is represented by its adjacency
matrix $A_p$. $A_p$ is a symmetric matrix whose entries are
\begin{equation}
\label{eq:ERprocess}
      A_{i,j}^{(p)} = \left\{   \begin{array}{ll}
                              0 & \text{if $i=j$} \\
                              \IND_{U_{i,j}<p} & \text{if $1\le i< j \le n$} \\
                               \IND_{U_{i,j}<p} & \text{if $1\le j< i \le n$}~,
                             \end{array} \right.
\end{equation}
where $(U_{i, j})_{1 \le i < j \le n}$ are independent random
variables, uniformly distributed on $[0, 1]$
and $\IND$ stands for the indicator function.
We call the family of random matrices $(A_p)_{p\in [0,1]}$ the \emph{Erd\H{o}s-R\'enyi random graph process}.

Spectral properties of adjacency matrices of random graphs have received considerable attention, see
F{\"u}redi and Koml{\'o}s \cite{FuKo81},
Krivelevich and Sudakov \cite{KrSu03},
Vu \cite{Vu05},
Erd\H{o}s, Knowles,  Yau,  and Yin \cite{ErKnYaYi13},
Benaych-Georges, Bordenave, and Knowles \cite{BeBoKn17a, BeBoKn17b},
Jung and Lee \cite{JuLe17},
Tran, Vu, and Wang \cite{TrVuWa13},
among many other papers.

In this paper we are primarily concerned with concentration properties of the spectral norm $\|A_p\|$
of the adjacency matrix. It follows from a general concentration inequality 
of Alon, Krivelevich, and Vu \cite{AlKrVu01}
for the largest eigenvalue
of symmetric random matrices with bounded independent entries that
for all $n\ge 1$, $p\in [0,1]$, and $t>0$,
\begin{equation}
\label{eq:alkrvu}
\PROB\left\{ \left| \|A_p\| - \EXP \|A_p\| \right| > t \right\}
\le 2e^{-t^2/32}~.
 \end{equation}
In particular,  $\var(\|A_p\|)\le C$ for a universal constant $C$. (One may take $C=16$, see \cite[Example 3.14]{BoLuMa13}.)
Krivelevich and Sudakov \cite{KrSu03} who studied the asymptotic value
of $\E\|A_p\|$ raised the question whether it is possible to 
improve (\ref{eq:alkrvu}).
As an application of our techniques we settle this question for non-sparse graphs. Moreover, we strengthen (\ref{eq:alkrvu}) in two different ways.

Our main result concerns the \emph{uniform} concentration of the spectral norm. 
In particular, first we prove that there exists a universal constant $C$ such that
\[
 \EXP  \sup_{p\ge C\log n/n}\left| \|A_p\| - \EXP \|A_p\| \right| \le C
\]
(see Theorem \ref{thm:unif} below).  
Informally, this result means that as we add new edges in
  the Erd\H{o}s-R\'enyi graph process, the value $\left|\|A_p\| - \EXP
    \|A_p\|\right|$ is never greater (up to an absolute constant
  factor) than the same value calculated for just one concrete random
  graph $G(n, \frac{1}{2})$. The proof of this result is based on an
  extension of the Dvoretzky-Kiefer-Wolfowitz (DKW) inequality (we
  refer to \cite{Mas90} for the state-of-the-art form) for particular
  functions of independent random variables.
For the entire
range $p\in [0,1]$, we are able to
prove a simple but slightly weaker inequality
\[
 \EXP  \sup_{p \in [0,1]}\left| \|A_p\| - \EXP \|A_p\| \right| \le
 C\sqrt{\log \log n}
\]
for a constant $C$ (Proposition \ref{prop:sqrtloglogn}). We also prove the tail bound of the form
\begin{equation}
\label{uniftailbound}
\PROB\left\{ \sup_{p\ge C\log n/n} \left| \|A_p\| - \EXP \|A_p\| \right| > t \right\}
\le e^{-t^2/C}~,
\end{equation}
which is a uniform version of the sub-Gaussian inequality (\ref{eq:alkrvu}) and has the same form up to absolute constant factors.  We leave open the question whether the restriction to the range $p\in
\left[C\frac{\log n}{n},1\right]$ is necessary for uniform
concentration. However, we also discuss very sparse regimes (i.e., when $p \ll \frac{1}{n}$).

Note that it follows from the Perron-Frobenius theorem that the spectral norm of $A_p$ equals the largest
eigenvalue of $A_p$, that is, $\|A_p\|=\lambda_p$. We use both
interchangeably throughout the paper, depending on the particular interpretation that is
convenient.

Our proofs hinge crucially on the so-called \emph{delocalization} property of
the eigenvector corresponding to the largest eigenvalue (see Erd\H{o}s, Knowles,  Yau,  and Yin \cite{ErKnYaYi13},
Mitra \cite{Mit09}), that is, the fact that the normalized eigenvector
corresponding to the largest eigenvalue is close, 
in a certain sense, to the vector $(1/\sqrt{n},\ldots,1/\sqrt{n})$. We
provide delocalization
bounds for the top eigenvector of $A_p$ tailored to our needs (Lemma
\ref{lem:weakdeloc})
and a uniform delocalization inequality (Lemma \ref{lem:unifdeloc}). An important fact is that some known delocalization bounds hold with probability $1 - \frac{C}{n^{\alpha}}$ (as in \cite{Mit09}) or with quasi-polynomial probability $1 - C\exp(-c(\log n)^{\beta})$ (see e.g. \cite{TrVuWa13} or Theorem 2.6 in \cite{ErKnYaYi13} ), where any choice of the parameter $\beta$ greater than zero is responsible for extra logarithmic factors, making these results not applicable in our case. So, to obtain tight concentration results we prove delocalization bounds which hold with the exponential probability of the form $1 - C\exp(-cnp)$ (up to logarithmic factors), which is significantly better in the regime when $p \gg \frac{\log n}{n}$.

As an application of our techniques, we prove sub-Gaussian inequalities for moments of $\|A_p\|$ of higher 
order (up to order approximately $np$). The precise statement is given in Theorem \ref{thm:moment} in 
Section \ref{sec:moment} below. In particular, we show that, for small values of $p$, $\|A_p\|$ is significantly more concentrated than 
what the general bound \eqref{eq:alkrvu} suggests. This technique implies, in particular, that there exists a universal constant $C$ such that
\begin{equation}
\label{varianceineqality}
   \var(\|A_p\|)\le Cp
\end{equation}
for all $n$ and $p\ge C\log^3 n/n$.

The rest of the paper is organized as follows. In Section
\ref{sec:results} we formalize and discuss the results of the paper.
The proofs are presented in Section \ref{sec:proofs}.
\section{Results}
\label{sec:results}

\subsection{Uniform concentration for the Erd\H{o}s-R\'enyi random graph process}

Next, we state our inequalities for the uniform concentration of the
spectral norm $\|A_p\|$---or, equivalently, for the largest
eigenvalue $\lambda_p$ of the adjacency matrix $A_p$ defined by
(\ref{eq:ERprocess}). Our first result shows that 

\begin{theorem}
\label{thm:unif}
There exists a constant $C$ such that, for all $n$,
\[
\E\sup_{p \in \left[\frac{64\log n}{n}, 1\right]}\left|\lambda_{p} - \E\lambda_{p}\right| \le C~.
\]
Moreover, for all $t \ge 2C$ ,
\[
\PROB\left\{\sup_{p \in \left[\frac{64\log n}{n}, 1\right]}\left|\lambda_{p} - \E\lambda_{p}\right| \ge t\right\} \le \exp(-t^2/128)~.
\]
\end{theorem}

For the numerical constant, our proof provides the (surely suboptimal)
value $C=5\times10^8$.
Our proof is based on the fact that the normalized
eigenvector corresponding to the largest eigenvalue of $A_p$ stays
close to the vector $(1/\sqrt{n},\ldots, 1/\sqrt{n})$. In Lemma
\ref{lem:unifdeloc} we prove an $\ell_2$ bound that holds uniformly 
over intervals of the form $[q,2q]$ when $q\in [4\log n/n, 1/2]$. 
It is because of the restriction of the range of $q$ in the uniform delocalization 
lemma that we need to impose $p\ge 64\log n/n$ in Theorem
\ref{thm:unif}. We do not know whether the uniform concentration
bound holds over the entire interval $p\in [0,1]$. However, we are
able to prove the following, only slightly weaker bound.

\begin{proposition}
\label{prop:sqrtloglogn}
There exists a constant $C'$ such that, for all $n$,
\[
\E\sup\limits_{p \in [0, 1]}\left|\lambda_{p} -
\E\lambda_{p}\right| \le C'\sqrt{\log \log n}~.
\]
\end{proposition}

The proof of Proposition \ref{prop:sqrtloglogn} uses direct approximation
arguments to handle the interval $p\in [0, 64\log n/n]$. In
particular, we show that
\[
\E\sup\limits_{p \in [0, 64\log n/n]}|\lambda_{p} - \E\lambda_{p}| \le 5\sqrt{16 + 2\log\log n}~,
\]
which, combined with Theorem \ref{thm:unif} implies Proposition \ref{prop:sqrtloglogn}.  As a second extension we consider the sparse regime when $p \ll \frac{1}{n}$.
\begin{proposition}
\label{prop:verysp}
Fix $k \in \mathbb{N}, k \ge 2$. There is a constant $C_k$ (its value may be extracted from the proof) which depends only on $k$ such that
\[
\E\sup\limits_{p \in [0, n^{-k/(k - 1)}]}\left|\lambda_{p} -
\E\lambda_{p}\right| \le C_k.
\]
\end{proposition}
\begin{remark}
A simple inspection of the proof of the concentration result of
Theorem \ref{thm:unif} shows that a tail inequality similar to the
second inequality of Theorem \ref{thm:unif} holds also for the range
$p \in [0, n^{-k/(k - 1)}]$. In this case the constant factors may depend on the choice of $k$.
\end{remark}

\subsection{Moment inequalities for the spectral norm}
\label{sec:moment}

As an application of our techniques we show that typical deviations of $\|A_p\|$ from its
expected value are of the order of $\sqrt{p}$. This is in accordance with the asymptotic normality
theorem of F{\"u}redi and Koml{\'o}s \cite{FuKo81}. However, while the result of \cite{FuKo81} holds
for fixed $p$ as $n\to \infty$, the theorem below is non-asymptotic. In particular, it holds for 
$p=o(1)$ as long as $np$ is at least of the order of $\log^3 (n)$.
 Note that the general non-asymptotic concentration inequality of \cite{AlKrVu01} only implies that typical
deviations are $O(1)$ and the question of possible improvements was raised in \cite{KrSu03}.

\begin{theorem} 
\label{thm:moment}
There exist constants $c,C$ such that for every  $k\in \left(2, \frac{c\left(\frac{\log(np)}{\log n}\right)^2p(n - 1)-\log(8(n-1))}{\log(\frac{1}{p})+\log(11^5/4)}\right]$,
\[
\left[\EXP\left( \|A_p\|- \EXP \|A_p\| \right)^k_+ \right]^{1/k} \le (Ckp)^{\frac{1}{2}}
\]
and
\[
\left[\EXP \left( \|A_p\|- \EXP \|A_p\| \right)^k_- \right]^{1/k} \le (C'kp)^{\frac{1}{2}}~.
\]
In particular, for some absolute constant $\kappa > 0$ it holds for all $n$ and $p \ge \kappa\log^3(n)/n$,
\[
   \var(\|A_p\|) \le Cp~.
\]

\end{theorem}
It is natural to ask whether the condition $p \ge \kappa\log^3(n)/n$ \footnote{Our analysis implies in fact a slightly better factor $\frac{\log^3(n)}{n(\log\log (n))^2}$ instead of $\frac{\log^3(n)}{n}$.}
is necessary. Although we believe that $\log^3(n)$ instead of the
lower powers of $\log n$ is only an artifact of our technique, the fact that the inequality $\var(\|A_p\|) \le Cp$
cannot hold for all values of $p$ is easily seen by taking $p=c/n^2$ for
a positive constant $c$. In this case, the probability that the graph $G(n,p)$
is empty is bounded away from zero. In that case $\|A_p\|=0$. On the 
other hand, with probability bounded away from zero, the graph
$G(n,p)$ contains a single edge, in which case $\|A_p\|=1$. Thus,
for  $p=c/n^2$, $\var(\|A_p\|) =\Omega(1)$, showing that the bound (\ref{eq:alkrvu}) is sharp in this range. 
Understanding the concentration properties of $\|A_p\|$ in the range
$n^{-2} \ll p \ll \log^3(n)/n$ is an intriguing open question.

The proof of Theorem \ref{thm:moment} is presented in Section
\ref{sec:momentproof}. The proof reveals that for the values of the
constants one may take $\kappa=2\times 835^2$, $C=966306$, $C'=1339945$, and $c=1/9408$. However,
these values have not been optimized. In the rest of this discussion we assume these numerical values. %

Using the moment bound with $k=t^2/(2Cp)$, Markov's inequality implies that for all
$0< t\le 2\sqrt{Cc}p\sqrt{n-1}\log(np)/(\log n\log(1/p))$,
\begin{equation}
\label{concentration}
    \PROB\left\{ |\|A_p\| - \EXP \|A_p\|| \ge t\right\}  \le 2^{-t^2/(2Cp)}~.
\end{equation}
This result improves \eqref{eq:alkrvu} in the regime when $t \ll p\sqrt{n}$ with some extra logarithmic factors and may be complemented by \eqref{eq:alkrvu} for the remaining values of $t$. Moreover, a simple inspection of the proof of Theorem \ref{thm:moment} shows that it may be extended in a way such that it is always not worse than the tail of \eqref{eq:alkrvu} for all $t \ge 0$.
The proof of this Theorem is based on 
general moment inequalities of Boucheron, Bousquet, Lugosi, and Massart \cite{BoBoLuMa04}  
(see also \cite[Theorems 15.5 and 15.7]{BoLuMa13}) that state that if $Z=f(X_1,\ldots,X_n)$ is a real random variable that
is a function of the independent random variables $X_1,\ldots,X_n$, then for all $k\ge 2$,
\begin{equation}
\label{eq:boboluma}
    \left[\EXP\left( Z- \EXP Z \right)^k_+ \right]^{1/k} \le \sqrt{3k} \left(\left[{V_{+}}^{k/2}\right] \right)^{1/k}~,
\end{equation}
and
\begin{equation}
\label{eq:bobolumalower}
\left[\EXP\left( Z- \EXP Z \right)^k_- \right]^{1/k}
\le
\sqrt{4.16 k}
\left(
\left(\EXP\left[ {V_{+}}^{k/2}\right] \right)^{1/k} \vee \sqrt{k} \left(\EXP \left[M^k\right]\right)^{1/k}\right)~,
\end{equation}
where $\vee$ denotes the maximum and the random variable $V_{+}$ is defined as
\[
   V_{+} = \EXP' \sum_{i=1}^n (Z-Z_i')_+^2~.
\]
Here $Z_i'=f(X_1,\ldots,X_{i-1},X_i',X_{i+1},\ldots,X_n)$ with $X_1',\ldots,X_n'$ being independent copies of $X_1,\ldots,X_n$
and $\EXP'$ denotes expectation with respect to $X_1',\ldots,X_n'$. Moreover,
\[
M= \max_i (Z-Z_i')_+~.
\]
Recall also that, by the Efron-Stein inequality (e.g., \cite[Theorem 3.1]{BoLuMa13})
\begin{equation}
\label{efstein}
   \var(Z) \le \EXP V_{+}~.
\end{equation}
The proof of Theorem \ref{thm:moment} is based on
(\ref{eq:boboluma}), applied  for the random variable $Z=\|A_p\|$.
In order to bound moments of the random variable $V_{+}$, we make use of the 
fact that the eigenvector of $A_p$ corresponding to the largest
eigenvalue is near the vector $(1/\sqrt{n},\ldots, 1/\sqrt{n})$. 
An elegant way of proving such results appears in Mitra \cite{Mit09}. We follow
Mitra's approach though we need to modify his arguments in order to achieve stronger probabilistic guarantees
for weak $\ell_{\infty}$ delocalization bounds. In Lemma
\ref{lem:weakdeloc} we provide the bound we need for the
proof of Theorem \ref{thm:moment}.

\section{Proofs}
\label{sec:proofs}

\subsection{Proof of Theorem \ref{thm:unif}}
\label{sec:unifproof}

We begin by noting that, if $p\le q$, then $A_q$ is element-wise
greater than or equal to $A_p$ and therefore $\|A_p\| \le \|A_q\|$ whenever $p \le q$.
(see Corollary $1.5$ in \cite{BePl94}).


We start with a lemma for the expected spectral norm for a sparse \ER graph.
Since the largest eigenvalue of the adjacency matrix is always bounded
by the maximum degree of the graph, $\EXP \|A_{\frac{1}{n}}\|$ is at most
of the order $\log n$. The next lemma improves this naive bound to
$O(\sqrt{\log n})$.
With more work, it is possible to improve the rate to
$\sqrt{\frac{\log n}{\log\log n}}$ 
(see the asymptotic result in \cite{KrSu03}). However, this slightly weaker version is sufficient for our purposes.

\begin{lemma}
\label{basic}
For all $n \ge 3$,
\[
\E\|A_{\frac{1}{n}}\| \le 173\sqrt{\log n}~. 
\]
\end{lemma}

\begin{proof}
First write
\[
\E\|A_{\frac{1}{n}}\| \le \E\|A_{\frac{1}{n}} - \E A_{\frac{1}{n}}\| + \|\E A_{\frac{1}{n}}\| \le \E\|A_{\frac{1}{n}} - \E A_{\frac{1}{n}}\| + 1~.
\]
Denote $B = A_{\frac{1}{n}} - \E A_{\frac{1}{n}}$ and let $B'$ be an
independent copy of $B$. Denoting by $\E'$ the expectation operator
with respect to $B'$, note that $\E' B' = 0$ and therefore, by
Jensen's inequality,
\[
\E\|B\| = \E\|B - \E'B'\| \le \E\|B - B'\|~.
\]
The matrix $B - B'$ is zero mean, its non-diagonal entries have
a symmetric distribution with variance $(2/n)(1 - 1/n)$ 
and all entries have absolute value bounded by $2$. 
Now, applying Corollary $3.6$ of Bandeira and van Handel \cite{Bava16} with $\alpha = 3$,
\[
\E\|B - B'\| \le e^{\frac{2}{3}}(2\sqrt{2} + 84\sqrt{\log n}) \le 6 +
166\sqrt{\log n}~.
\]
Thus,
\[
\E\|A_{\frac{1}{n}}\| \le 7 + 166\sqrt{\log n} \le 173\sqrt{\log n}~.
\]
\end{proof}


The next lemma and the uniform delocalization inequality of Lemma
\ref{lem:unifdeloc} (presented in Section \ref{sec:deloc}) are the crucial building blocks of the proof of
Theorem \ref{thm:unif}.

\begin{lemma}
\label{lem:centered}
For all $n$ and $q \in [\log n/n, \frac{1}{2}]$, 
\[
\PROB\left\{ \sup_{p \in [q, 2q]}\|A_p - \E A_p\| > 420\sqrt{nq}
\right\} \le e^{-nq/64}~.
\]
\end{lemma}
\begin{proof}
By (\ref{eq:alkrvu}), for each fixed $p$ and
for all $t>0$, we have
\[
\PROB\left\{ \|A_p-\EXP A_p\| - \EXP \|A_p-\EXP A_p\| > t \right\}
\le e^{-t^2/32}~.
\]
On the other hand, using the same symmetrization trick as in Lemma
\ref{basic}, Corollary 3.6  of Bandeira, van Handel \cite{Bava16} 
implies that
for any $p \ge \log n/n$,
\begin{equation}
\label{bandeira}
\EXP \|A_p-\EXP A_p\| \le e^{\frac{2}{3}}(2\sqrt{2np} + 84\sqrt{\log n}) \le 170\sqrt{np}~.
\end{equation}
These two results imply
\[
\PROB\left\{ \|A_p-\EXP A_p\| > 172\sqrt{np} \right\}
\le e^{-np/8}~.
\]
Let now $q\ge \log n/n$ and for $i=0,1,\ldots,\lceil nq\rceil$,
define $p_i=q+i/n$. Then 
\begin{eqnarray*}
\sup_{p\in [p_i,p_{i+1}]} \left( \|A_p-\EXP A_p\| - \|A_{p_i}-\EXP
  A_{p_i}\| \right)
& \le &
\sup_{p\in [p_i,p_{i+1}]} \left( \|A_p-A_{p_i}\| + \|\EXP A_p-\EXP A_{p_i}\| \right)
\\
& = &
\sup_{p\in [p_i,p_{i+1}]} \left( \|A_p-A_{p_i}\| + \|\EXP A_{p-p_i}\| \right) \\
& = &
 \|A_{p_{i + 1}}-A_{p_i}\| + \|\EXP A_{1/n}\|  \\
 & \le &
 \|A_{p_{i + 1}}-A_{p_i}\| + 1 \\
& = &
\EXP \|A_{1/n}\| + \left( \|A_{p_{i + 1}}-A_{p_i}\|- \EXP \|A_{p_{i + 1}}-A_{p_i}\|\right) + 1
\\
& \le &
1 + 173\sqrt{\log n} + \sqrt{nq}\\
&\le & 176\sqrt{nq}
\end{eqnarray*}
with probability at least $1-e^{-nq/32}$, where we used Lemma
\ref{basic} and (\ref{eq:alkrvu}).
Thus, by the union bound, with probability at least
$1- nq e^{-nq/32}- nq e^{-np/8}
\ge 1-e^{-nq/64}$,
\begin{align*}
\sup_{p\in [q,2q]} \|A_p-\EXP A_p\|
&\le
\max_{i\in \{0,\ldots,\lceil nq\rceil\}} \|A_{p_i}-\EXP A_{p_i}\| + 176\sqrt{nq}
\\
& \le 172\sqrt{2nq} + 176\sqrt{nq} 
\\
&\le 420\sqrt{nq}.
\end{align*}
as desired.
\end{proof}

\medskip
\noindent
{\bf Proof of Theorem \ref{thm:unif}.}
Denote by $\ol{1}\in \R^n$ the vector whose components are all equal
to $1$. Let $B^n_2=\{x\in \R^n: \|x\|_2 \le 1\}$ be the unit Euclidean
ball. Define
the event $E_{1}$ that $v_{p} \in \frac{\ol{1}}{\sqrt{n}} +
\frac{2896}{\sqrt{np}}B^n_{2}$ for all $p \in \left[64\log n/n,1\right]$.
By Lemma \ref{lem:unifdeloc} (see Section \ref{sec:deloc} below), for $n\ge 7$,
\[
\PROB\{E_1\} \ge 1 - 4\sum\limits_{j = 0}^{\infty}\exp\left(-2^j \log n\right) \ge 1 -
4\sum\limits_{j = 0}^{\infty}\left(\frac{1}{n}\right)^{2^j} \ge 1 -
\frac{4}{n}\sum\limits_{j = 0}^{\infty}\left(\frac{1}{7}\right)^{j} =
1 - \frac{32}{7n}~.
\]
Now define the event $E_2$ that for all $p \in \left[\frac{64\log n}{n},
1\right]$, $\|A_p - \E A_{p}\| \le 420\sqrt{2np}$.
Similarly to the calculation above, by Lemma \ref{lem:centered}, $\PROB\{E_2\} \ge 1 - \frac{32}{7n}$.

Denoting by $S^{n-1}=\{x\in \R^n: \|x\|_2=1\}$ the Euclidean unit sphere in
$\R^n$, define
\[
\ol\lambda_{p}= \sup_{x \in S^{n - 1}} x^TA_px\IND_{E_1 \cap E_2} \quad
\text{and} \quad \ol{A}_{p}=A_{p}\IND_{E_2}~.
\]
Then we may write the decomposition 
\[
\ol\lambda_{p}=
\sup_{x \in \frac{\ol{1}}{\sqrt{n}} +
  \frac{2896}{\sqrt{np}}B^n_{2}} x^T\ol{A}_px =
\frac{\ol{1}}{\sqrt{n}}\ol{A}_p\frac{\ol{1}}{\sqrt{n}} + 2\sup_{z
  \in \frac{2896}{\sqrt{np}}B^n_{2}} z^{T}\ol{A}_p\left(\frac{\ol{1}}{\sqrt{n}}
+ \frac{z}{2}\right)~.
\]
 Then
\begin{eqnarray}
\label{eq:lambdaprime}
\lefteqn{
\sup_{p \in \left[\frac{64\log n}{n}, 1\right]}|\ol\lambda_{p} - \E\ol\lambda_{p}|
  } \nonumber \\
  &\le &
2\sup_{p \in \left[\frac{64\log n}{n}, 1\right]}\left|\sup_{z \in \frac{2896}{\sqrt{np}}B^n_{2}}(z^{T}\ol{A}_p(\frac{\ol{1}}{\sqrt{n}} + \frac{z}{2}))- \E\sup_{z \in \frac{2896}{\sqrt{np}}B^n_{2}}(z^{T}\ol{A}_p(\frac{\ol{1}}{\sqrt{n}} + \frac{z}{2}))\right|
\nonumber \\
& & + \sup_{p \in \left[\frac{64\log n}{n}, 1\right]}\left|{\frac{\ol{1}}{\sqrt{n}}}^T\ol{A}_p{\frac{\ol{1}}{\sqrt{n}}}- \E\frac{\ol{1}^{T}}{\sqrt{n}}\ol{A}_p\frac{\ol{1}}{\sqrt{n}}\right|. 
\end{eqnarray}
For the second term on the right-hand side of (\ref{eq:lambdaprime}), since $A_p - \ol{A}_p = A_p\IND_{\overline{E}_{2}}$ we have 
\begin{eqnarray*}
\lefteqn{
\E\sup_{p \in \left[\frac{64\log n}{n},
  1\right]}\left|{\frac{\ol{1}}{\sqrt{n}}}^T\ol{A}_p{\frac{\ol{1}}{\sqrt{n}}}-
  \E\frac{\ol{1}^{T}}{\sqrt{n}}\ol{A}_p\frac{\ol{1}}{\sqrt{n}}\right|  } 
\\
&\le &\E\sup_{p \in \left[\frac{64\log n}{n}, 1\right]}\left|{\frac{\ol{1}}{\sqrt{n}}}^TA_p{\frac{\ol{1}}{\sqrt{n}}}- \E\frac{\ol{1}^{T}}{\sqrt{n}}A_p\frac{\ol{1}}{\sqrt{n}}\right| + 2nP(\overline{E}_{2})~.
\end{eqnarray*}
Note that $\frac{\ol{1}}{\sqrt{n}}^TA_p\frac{\ol{1}}{\sqrt{n}} =
(2/n)\sum_{i<j} \IND_{U_{i,j}<p}$. Thus, the first term on the
right-hand side is just the maximum deviation between the 
cumulative distribution function of a uniform random variable and its
empirical counterpart based on $\binom{n}{2}$ random samples. This
may be bounded by the classical Dvoretzky-Kiefer-Wolfowitz theorem \cite{DvKiWo56}.
Indeed, by Massart's version \cite{Mas90},
we have 
\begin{eqnarray*}
\E\sup_{p  \in \left[\frac{64\log n}{n}, 1\right]}\left|{\frac{\ol{1}}{\sqrt{n}}}^TA_p{\frac{\ol{1}}{\sqrt{n}}}-
  \E\frac{\ol{1}^{T}}{\sqrt{n}}A_p\frac{\ol{1}}{\sqrt{n}}\right| 
& \le &
\E\sup_{p \in [0, 1]}\left|{\frac{\ol{1}}{\sqrt{n}}}^TA_p{\frac{\ol{1}}{\sqrt{n}}}-
  \E\frac{\ol{1}^{T}}{\sqrt{n}}A_p\frac{\ol{1}}{\sqrt{n}}\right| \\
& \le &
4\int\limits_{t = 0}^{\infty}\exp(-2t^2)dt = \sqrt{2\pi}~.
\end{eqnarray*}
Thus, the second term on the right-hand side of (\ref{eq:lambdaprime}) is
bounded by the absolute constant $\sqrt{2\pi} +\frac{64}{7} \le 12$
since $P(\overline{E}_{2}) \le \frac{32}{7n}$. 

In order to bound the first term on
the right-hand side of (\ref{eq:lambdaprime}), we write
\begin{eqnarray*}
\lefteqn{
\sup_{p \in \left[\frac{64\log n}{n}, 1\right]}\left|\sup_{z \in
  \frac{2896}{\sqrt{np}}B^n_{2}} z^{T}\ol{A}_p\left(\frac{\ol{1}}{\sqrt{n}} +
  \frac{z}{2}\right)- \E\sup_{z \in
  \frac{2896}{\sqrt{np}}B^n_{2}} z^{T}\ol{A}_p\left(\frac{\ol{1}}{\sqrt{n}} +
  \frac{z}{2}\right)\right|   } \\
\\
&\le& \sup_{p \in \left[\frac{64\log n}{n}, 1\right]}\sup_{z \in
      \frac{2896}{\sqrt{np}}B^n_{2}}\left|
      z^{T}\ol{A}_p\left(\frac{\ol{1}}{\sqrt{n}} + \frac{z}{2}\right)- \E z^{T}\ol{A}_p\left(\frac{\ol{1}}{\sqrt{n}} + \frac{z}{2}\right)\right|
\\
&\le& \sup_{p \in \left[\frac{64\log n}{n},
      1\right]}\frac{2896}{\sqrt{np}}\sup_{z \in
      \frac{2896}{\sqrt{np}}B^n_{2}}
\left\|\frac{\ol{1}}{\sqrt{n}} + \frac{z}{2}\right\|\ \cdot \|\ol{A}_{p} -\E \ol{A}_p\| 
\\
& \le & \sup_{p \in \left[\frac{64\log n}{n}, 1\right]}2896\times 594\left(1 + \frac{1448}{\sqrt{np}}\right)
\\
& \le & 2896\times 594\left(1 + \frac{1448}{\sqrt{64\log(7)}}\right) \le 4.5\times10^8~.
\end{eqnarray*}
Finally, note that with probability at least $1 - \frac{64}{7n}$ for
all $p \in \left[\frac{64\log n}{n}, 1\right]$ we have $\ol\lambda_{p}
= \lambda_p$. 
Moreover, for all $p$,
\[
\E\lambda_{p} - \E\lambda'_p \le \E\sup_{x \in S^{n - 1}}(x^TA_px(1 -
\IND_{E_1 \cap E_2})) \le nP(\overline{E}_{1} \cup \overline{E_2}) \le
\frac{64}{7}~.
\]
Thus,
\[
\E\sup_{p \in \left[\frac{64\log n}{n}, 1\right]}|\lambda_{p} - \E\lambda_{p}| \le \frac{128}{7} + \E\sup_{p \in \left[\frac{64\log n}{n}, 1\right]}|\ol\lambda_{p} - \E\ol\lambda_{p}| \le 5\times10^8~,
\]
proving the first inequality of the theorem.

To prove the second inequality, we follow the argument of Example
$6.8$ in \cite{BoLuMa13}. 
Denote $Z = \sup_{p \in \left[\frac{64\log n}{n},
    1\right]}|\lambda_{p} - \E\lambda_{p}|$ and $Z_{i, j}' = \sup_{p
  \in \left[\frac{64\log n}{n}, 1\right]}|\lambda'_{p} -
\E\lambda_{p}|$ where $\lambda'_{p}$ is the largest eigenvalue of the
adjacency matrix $A'_{p}$ of the random graph that is obtained from
$A_p$ by replacing $U_{i,j}$ by an independent copy. 
Denoting the first eigenvector of $A_p$ by $v_p$ and the first eigenvector of $A'_p$ by $v'_p$ and the (random) maximizer $\sup_{p \in \left[\frac{64\log n}{n}, 1\right]}|v^T_pA_pv_p - \E\lambda_{p}|$ by $p^*$, we have
\begin{eqnarray*}
(Z - Z_{i, j}')_{+} &\le &
\left(\sup_{p \in \left[\frac{64\log n}{n}, 1\right]}|v^T_pA_pv_p - \E\lambda_{p}| - \sup_{p \in \left[\frac{64\log n}{n}, 1\right]}|v'^T_pA'_pv'_p - \E\lambda_{p}|\right)\IND_{Z \ge Z_{i, j}'}
\\
&\le & \left|v^T_{p^*}A_{p^*}v_{p^*} - \E\lambda_{p^*} - v'^T_{p^*}A'_{p^*}v'_{p^*} - \E\lambda_{p^*}\right|\IND_{Z \ge Z_{i, j}'}
\\
&\le & \left|v^T_{p^*}(A_{p^*} - A'_{p^*})v_{p^*}\right|\IND_{Z \ge Z_{i, j}'}
\\
&\le & 4|v^i_{p^*}v^j_{p^*}|~.
\end{eqnarray*}
This implies $\sum\limits_{1 \le i \le j \le n}(Z - Z_{i, j}')^2_{+} \le 16$. Thus, for any $t \ge 0$,
\[
\PROB\left\{\sup_{p \in \left[\frac{64\log n}{n}, 1\right]}|v^T_pA_pv_p - \E\lambda_{p}| - \E \sup_{p \in \left[\frac{64\log n}{n}, 1\right]}|v^T_pA_pv_p - \E\lambda_{p}| \ge t\right\} \le \exp(-t^2/32)~.
\]
Using the bound $ \E \sup_{p \in \left[\frac{64\log n}{n}, 1\right]}|v^T_pA_pv_p - \E\lambda_{p}| \le 5\times10^8$, we have for $t' = t + 5\times10^8$
\[
\PROB\left\{\sup_{p \in \left[\frac{64\log n}{n}, 1\right]}|v^T_pA_pv_p - \E\lambda_{p}| \ge t'\right\} \le  \exp(-(t' - 5\times10^8)^2/32)~.
\]
For $t' \ge 10^9$ the claim follows.

\subsection{Proof of Theorem \ref{thm:moment}}
\label{sec:momentproof}

Let $v_p$ denote an eigenvector corresponding to the largest eigenvalue of $A_p$ such that $\|v_p\|=1$.
Recall that $\kappa=2\times 835^2$ and $c=1/9408$. One of the key
elements of the proof is the following new variant of a 
delocalization inequality of Mitra \cite{Mit09}.

\begin{lemma}
\label{lem:weakdeloc}
Let $n \ge 7$ and $p \ge \kappa\log^3(n)/n$. 
Let $v_p$ denote an eigenvector corresponding to the largest
eigenvalue $\lambda_p$ of $A_p$ with $\|v_p\|_2=1$. 
Then, with probability at least 
\begin{eqnarray*}
 1 - 4(n - 1)\exp\left(-2c\left(\frac{\log(np)}{\log n}\right)^2(n - 1)p\right)~,
\end{eqnarray*}
\[
\left\|v_p\right\|_{\infty} \le \frac{11}{\sqrt{n}}~.
\]
\end{lemma}

The lemma is proved in Section \ref{sec:deloc} below.
Based on this lemma, we may prove Theorem \ref{thm:moment}:

\medskip
\noindent
{\bf Proof of Theorem \ref{thm:moment}.}
We apply (\ref{eq:boboluma}) for the random variable $Z=\|A_p\|$, as a function of the $\binom{n}{2}$ independent
Bernoulli random variables $A_{i,j}=A_{i,j}^{(p)}$, $1\le i < j \le n$.
Let $E_1$ denote the event $\|v_p\|_{\infty} \le 11/\sqrt{n}$. By Lemma \ref{lem:weakdeloc},
\[
   \PROB\{E_1\} \ge 1 - 4(n - 1)\exp\left(-\frac{1}{4704}\left(\frac{\log(np)}{\log n}\right)^2(n - 1)p\right)~.
\]
For $1\le i < j \le n$, denote by $\lambda'_{i,j}$ the largest eigenvalue of the adjacency matrix obtained by replacing $A_{i, j}$ 
(and $A_{j, i}$ ) by an independent copy $A_{i,j}'$ and keeping all other entries unchanged. 
If the components of the eigenvector $v_p$ (corresponding to the eigenvalue $\lambda_p$) are $(v^{1}_p,\ldots,v^{j}_p)$,
then
\[
V_{+} = \E'\sum\limits_{i < j}^{n}(\lambda_p - \lambda'_{i,j})_{+}^2 \le 4\sum\limits_{i < j}^{n}\E' \left[(v^{i}_{p})^2(v^{j}_{p})^2(A_{i, j} - A'_{i, j})^2 \right] = 4\sum\limits_{i < j}^{n}(v^{i}_{p})^2(v^{j}_{p})^2(p + (1 - 2p)A_{i, j})_{+}~.
\]
Since $(A_{i, j} - A'_{i, j})^2\le 1$ and $\sum_i^n (v^{i}_p)^2=1$, we always have $V_{+}\le 4$. 
On the event $E_1$, we have a better control:
\[
V_{+}\IND_{E_1} \le  \frac{4 \cdot 11^4}{n^2}  \left(\binom{n}{2} p + (1 - 2p)\sum\limits_{i < j}A_{i, j}\right)~.
\]
Let $E_2$ denote the event that $\sum\limits_{i < j}^{n} A_{i, j} \le 2\E\sum\limits_{i < j}^{n} A_{i, j} \le pn(n-1)$.
By Bernstein's inequality,  $\PROB\{E_2\} \ge 1 - \exp(-\frac{3pn(n-1)}{8})$.
Then 
\[
V_{+}\IND_{E_1\cap E_2} \le  11^5p~.
\]
Thus,
\begin{eqnarray*}
\EXP\left[ (V_{+})^{\frac{k}{2}}\right] 
& = &
\EXP\left[ (V_{+})^{\frac{k}{2}}\IND_{E_1\cap E_2}\right] +  
\EXP\left[ (V_{+})^{\frac{k}{2}}\left(\IND_{\overline{E_1}}+ \IND_{\overline{E_2}}\right)\right] \\
& \le &
\left(11^5p \right)^{k/2}
+  
4^{k/2} \left(\PROB\{\overline{E_1}\} + \PROB\{\overline{E_2}\}\right)
\\
& \le &
2\left(11^5p \right)^{k/2}~,
\end{eqnarray*}
whenever $\PROB\{\overline{E_1}\} + \PROB\{\overline{E_2}\} \le \left(11^5p/4 \right)^{k/2}$.
This holds whenever 
\[
8(n - 1)\exp\left(-\frac{1}{4704}\left(\frac{\log(np)}{\log n}\right)^2(n - 1)p\right)
\le \left(11^5p/4 \right)^{k/2}~,
\]
guaranteed by our assumption on $k$.
The proof of the bound for the upper tail follows from (\ref{eq:boboluma}). 
The bound for the variance follows from the Efron-Stein inequality \eqref{efstein}. 

For the bound for the lower tail we use (\ref{eq:bobolumalower}).
Note that
\[
\max_{i<j} (\lambda_p - \lambda'_{i,j})_{+}\IND_{E_1}  \le 2\max_{i<j}(v_p^iv_p^j(A_{i, j} - A'_{i, j}))_{+}\IND_{E_1}  \le \frac{72}{n}~,
\]
and therefore
\[
\E\max_{i < j}(v_p^iv_p^j(A_{i, j} - A'_{i, j}))^k_{+}\IND_{E_1}  \le \left(\frac{72}{n}\right)^k~.
\]
Moreover,
\[
\E\max_{i < j}(2v_p^iv_p^j(A_{i, j} - A'_{i, j}))^k_{+}\IND_{\overline{E_1}}
\le 2^k\PROB\left\{\overline{E_1}\right\}
\le 2^{k + 2}(n - 1)\exp\left(-\frac{1}{4704}\left(\frac{\log(np)}{\log(n)}\right)^2(n - 1)p\right)~.
\]
We require 
\[
\left(\frac{72}{n}\right)^k \ge 2^{k + 2}(n - 1)\exp\left(-\frac{1}{4704}\left(\frac{\log(np)}{\log(n)}\right)^2(n - 1)p\right)
\]
which holds whenever
\[
k \le \frac{\frac{1}{4704}\left(\frac{\log(np)}{\log(n)}\right)^2(n - 1)p - \log(4(n - 1))}{\log(\frac{n}{36})}~.
\]
Under this condition
\[
\left(\E\max_{i < j}(v_p^iv_p^j(A_{i, j} - A'_{i, j}))^k_{+}\right)^{\frac{1}{k}} \le \frac{144}{n}~.
\]
Under our conditions for $k$ and $p$, we have $k(144/n)^2\le 2\cdot 11^5p$
and therefore (\ref{eq:bobolumalower}) implies the last inequality of
Theorem \ref{thm:moment}.

\begin{remark}{}
  It is tempting to understand if different approaches may lead to a simplified proof of Theorem
  \ref{thm:moment} with the weaker condition of $p \ge \frac{\log
    n}{n}$. 
  Perturbation theory based approach has been used by \cite{ErKnYaYi13} for
  the analysis of concentration of $\|A_p\|$ around its
  expectation. To compare with this paper, in this remark we assume
  that $A_p$ is the adjacency matrix of an
  Erd\H{o}s-R\'enyi random graph with loops, that is, all
  vertices link to themselves, each with probability $p$. Our results
  may be adapted to this case in a straightforward manner via minor
  changes in the constant factors. It can be shown (see formula in (6.17)
  in Section 6 of \cite{ErKnYaYi13}) that when
  $\|A_p - \E A_p\| < \|A_p\|$,
\begin{equation}
\label{lambda_sum}
\|A_p\| = \sum\limits_{j= 0}^{\infty}p\ol{1}^T\left(\frac{A_p - \E A_p}{\lambda_p}\right)^j\ol{1}~,
\end{equation}
where $\ol{1}\in \R^n$ is the vector whose components are all equal
to $1$.
Theorem 6.2 in \cite{ErKnYaYi13} (which is based on a thorough
analysis of the sum \eqref{lambda_sum}) shows that, for any $\xi \in
[2, A_0\log\log(n)]$, provided that $\frac{pn}{1 - p} \ge C_0^2\log^{4
  \xi} (n)$,  we have, with probability at least $1 - \exp\left(-\nu \log^{\xi}(n)\right)$,
\begin{equation}
\label{devap}
\|A_p\| = \E\|A_p\| + \frac{\ol{1}^T(A_p - \E A_p)\ol{1}}{n} + O\left(\frac{\log^{2\xi} (n)}{(1 - p)\sqrt{n}}\right)~,
\end{equation}
where the constant factors in the $O$-notation may depend on $\xi$,
and $\nu, A_0 \ge 10$ are absolute constants. It can be easily seen
that, up to an absolute constant factor, this bound implies the variance
bound \eqref{varianceineqality} but only in the regime
$p \ge \frac{c_0\log^8 (n)}{n}$, where $c_0$ is an absolute
constant. Moreover, it appears that the probability with which
\eqref{devap} holds is not sufficient to recover Theorem
\ref{thm:moment} in a straightforward manner. Indeed, we know that
\eqref{devap} does not hold on the event $E$ with
$\PROB\{E\} \le \exp\left(-\nu \log^{\xi}(n)\right)$. Let us consider
the moments of $\|A_p\|$ when $E$ holds. It can be shown using
\eqref{eq:alkrvu} that for some absolute $C >0$
\[
\E(\|A_p\| - \E\|A_p\|)^k\ind_{E} \le \sqrt{\E(\|A_p\| - \E\|A_p\|)^{2k}\PROB\{E\}} \le (Ck)^{\frac{k}{2}}\exp\left(-\nu \log^{\xi}(n)/2\right)~.
\]
To get the same bound as in Theorem \ref{thm:moment} we need $(Ck)^{\frac{k}{2}}\exp\left(-\nu \log^{\xi}(n)/2\right) \le (C^{\prime}kp)^{\frac{k}{2}}$, which holds when 
\[
k \le \frac{2\nu\log^{\xi}(n)}{\log\left(\frac{C}{C'p}\right)}~.
\]
The last inequality is more restrictive than what is required in
Theorem \ref{thm:moment} when
$p \ge \frac{c\nu \log^{\xi + 2}(n)}{n\log^{2}(np)}$ for some absolute
constant $c > 0$. To sum up, compared to \eqref{devap} our Theorem
\ref{thm:moment} has a different proof and provides tighter results in
some natural situations.
\end{remark}

\subsection{Delocalization bounds}
\label{sec:deloc}

In this section we prove the ``delocalization'' inequalities that
state that the eigenvector $v_p$ corresponding to the largest
eigenvalue of $A_p$ is close to the ``uniform'' vector $n^{-1/2}\ol{1}$.
The following lemma is crucial in the proof of Theorem \ref{thm:unif}.
This proof is based on an argument of Mitra \cite{Mit09}. However, we
need to modify it to get uniformity and also significantly better concentration guarantees.

\begin{lemma}
\label{lem:unifdeloc}
Let $n \ge 7$ and $q \in [\frac{4\log n}{n}, \frac{1}{2}]$. Then,
with probability 
 $1 - 4\exp(-nq/64)$,
\[
\sup_{p \in [q, 2q]}\left\|v_p - \frac{\ol{1}}{\sqrt{n}}\right\|_{2} \le \frac{2896}{\sqrt{nq}}~.
\]
\end{lemma}
\begin{proof}

\noindent
First note that there exists a unique vector $v^{\perp}_p$ with
$(v^{\perp}_p, v_p) = 0$ and $\|v^{\perp}_p\|_2 = 1$ 
such that
\begin{equation}
\label{form}
\ol{1}/\sqrt{n} = \alpha v_p + \beta v^{\perp}_p
\end{equation}
for some $\alpha,\beta\in \R$. 
By Lemma \ref{lem:centered}, with probability at least $1 - \exp(-nq/64)$,
\[
\sup_{p \in [q, 2q]}\|A_p - \E A_p\| \le 420\sqrt{nq}~.
\]
Notice that $\E A_p =
pn\frac{\ol{1}}{\sqrt{n}}\frac{\ol{1}^T}{\sqrt{n}} - pI_{n}$, where
$I_{n}$ is an identity $n \times n$ matrix.  
Since the graph with adjacency matrix $A_q$ is connected with probability  at least $1 - (n -
1)\exp(-nq/2)$ (see, e.g., \cite[Section 5.3.3]{Tro15}), by
monotonicity of the property of connectedness, the same
holds simultaneously for all graphs $A_p$ for $p\in [q,2q]$. 
Also, by the Perron-Frobenius theorem, if the graph is connected, the components of $v_p$ are all
nonnegative for all $p \in [q, 2q]$. 
Using that $\alpha = \left(\frac{\ol{1}}{\sqrt{n}}, v_p\right)$,
\begin{eqnarray*}
(A_p - \E A_p)v_p &=& \lambda_p v_p - pn\frac{\ol{1}}{\sqrt{n}}\frac{\ol{1}^T}{\sqrt{n}}v_p + pv_{p}
\\
&=& \lambda_p v_p - pn\alpha\frac{\ol{1}}{\sqrt{n}} + pv_p
\\
&= & \lambda_p v_p - pn\alpha(\alpha v_p + \beta v^{\perp}_p) + pv_p
\\
&= & (\lambda_p + p - pn\alpha^2)v_p - pn\alpha\beta v^{\perp}_p.
\end{eqnarray*}
This leads to 
\begin{equation}
\label{cond}
(\lambda_p + p - pn\alpha^2)^2 \le 420^2nq~.
\end{equation}
Since $\alpha \in [0, 1]$, this implies that, with probability at
least $1 - \exp(-nq/64) - (n - 1)\exp(-nq/2)$, 
simultaneously for all $p \in [q, 2q]$
\begin{equation}
\label{lambdaupper}
\lambda_p \le p(n - 1) + 420\sqrt{nq}~.
\end{equation}
We may get a lower bound for $\lambda_p$ by noting that
\[
\lambda_p \ge \frac{1}{n}\ol{1}^TA_p\ol{1}  = \frac{2}{n}\sum\limits_{i < j}^n\IND_{U_{ij} < p}~. 
\]
Applying Massart's version of the Dvoretzky-Kiefer-Wolfowitz theorem \cite{Mas90}, we have, for all $t \ge 0$,
\[
\PROB\left\{\sup_{p \in [0, 1]}\left|\frac{2}{n}\sum\limits_{i <
      j}^n\IND_{U_{ij} < p} - (n - 1)p\right| \ge (n - 1)t\right\} 
\le 2\exp\left(-n(n - 1)t^{2}\right)~.
\]
Choosing $t = \frac{\sqrt{nq}}{n - 1}$,  we have, with probability at least $1 - 2\exp\left(-nq/2\right)$,
\begin{equation}
\label{lambdalow}
\lambda_p \ge p(n - 1) - \sqrt{nq}~.
\end{equation}
This lower bound, together with (\ref{cond}) gives
\begin{equation}
\label{weakres}
\alpha \ge \alpha^2 \ge \frac{\lambda_p + p}{pn} - \frac{420\sqrt{nq}}{pn} \ge 1 - \frac{421}{\sqrt{nq}}
\end{equation}
with probability at least $1 - \exp(-nq/64) - (n - 1)\exp(-nq/2) -
2\exp\left(-nq/2\right) \ge 1 - 4(n - 1)\exp(-nq/64)$.
For the rest of the proof, we denote this event by $E$.

Next, write
\begin{equation}
\label{eq:triangle}
\left\|\frac{\ol{1}}{\sqrt{n}} - v_{p}\right\|_{2} \le \left\|\frac{A_p}{\lambda_{p}}\frac{\ol{1}}{\sqrt{n}} - v_{p}\right\|_{2} + \left\|\frac{A_p}{\lambda_{p}}\frac{\ol{1}}{\sqrt{n}} - \frac{\ol{1}}{\sqrt{n}}\right\|_{2}.
\end{equation}
We analyze both terms on the right-hand side.
Observe that $\E A_p \frac{\ol{1}}{\sqrt{n}} =
\frac{(n - 1)p\ol{1}}{\sqrt{n}}$.
The second term on the right-hand side of (\ref{eq:triangle}) may be
bounded on the event $E$, for all $p\in [q,2q]$, as
\begin{align*}
\left\|\frac{A_p}{\lambda_{p}}\frac{\ol{1}}{\sqrt{n}} - \frac{\ol{1}}{\sqrt{n}}\right\|_{2} &\le \frac{1}{\lambda_p}\left\|A_p\frac{\ol{1}}{\sqrt{n}} - \frac{(n - 1)p\ol{1}}{\sqrt{n}}\right\|_{2} + \frac{1}{\lambda_p}\left\|\frac{((n - 1)p - \lambda_p)\ol{1}}{\sqrt{n}}\right\|_{2}
\\
&= \frac{1}{\lambda_p}\left\|A_p\frac{\ol{1}}{\sqrt{n}} - \E A_p\frac{\ol{1}}{\sqrt{n}}\right\|_{2} + \frac{|(n - 1)p - \lambda_p|}{\lambda_p}
\\
&\le \frac{\left\|A_p - \E A_p\right\| + |(n - 1)p - \lambda_p|}{\lambda_p}
\\
&\le \frac{420\sqrt{nq} + 420\sqrt{nq}}{p(n - 1) - \sqrt{nq}}
\\
&\le \frac{1640}{\sqrt{nq}}~.
\end{align*}

Thus, on the event $E$, for all $p \in [q, 2q]$,
\[
\left\|\frac{\ol{1}}{\sqrt{n}} - v_{p}\right\|_{2} 
\le \left\|\frac{A_p}{\lambda_{p}}\frac{\ol{1}}{\sqrt{n}} - v_{p}\right\|_{2} + \frac{1640}{\sqrt{nq}}~.
\]
For each $p$, we may write $v^{\perp}_p = \sum\limits_{i = 2}^{n}\gamma_{i}v^i_p$, where $v^i_p$ is the $i$-th orthonormal eigenvector of $A_p$.
Then
\[
\frac{A_p}{\lambda_{p}}\frac{\ol{1}}{\sqrt{n}} = \alpha v_p + \beta\sum\limits_{i = 2}^{n}\frac{\gamma_{i}\lambda_{i}v_p^i}{\lambda_p}~,
\]
where $\lambda_{i}$ is $i$-th eigenvalue of $A_p$. 
By the Perron-Frobenius theorem, we have $|\lambda_{i}| \le
\lambda_{p}$ for all $i = 2, \ldots, n$. 
Moreover, from F\"uredi and Koml\'os \cite[Lemmas 1 and 2]{FuKo81} ,
for all $t \in \mathbb{R}$ we have that $|\lambda_{i}| \le \|A_{p} -
t\frac{\ol{1}}{\sqrt{n}}\frac{\ol{1}^T}{\sqrt{n}}\|$ for $i \ge 2$. 
Choosing $t = np$ we obtain $|\lambda_{i}| \le \|A_{p} - \E A_p\| +
p\|I_{n}\| \le 420\sqrt{nq} + p \le 422\sqrt{nq}$~.
Thus, using (\ref{weakres}), on the event $E$,
\[
\left\|\frac{A_p}{\lambda_{p}}\frac{\ol{1}}{\sqrt{n}} - v_{p}\right\|_{2} \le 1 - \alpha + \beta\max\limits_{i \ge 2}\frac{|\lambda_i|}{\lambda_p} + \frac{1640}{\sqrt{nq}}\le \frac{2061}{\sqrt{nq}} + \frac{422\sqrt{nq}}{(n - 1)p - \sqrt{nq}} \le \frac{2896}{\sqrt{nq}}~,
\]
as desired.
\end{proof}

We close this section by proving the ``weak'' delocalization bound of Lemma \ref{lem:weakdeloc}.

\medskip
\noindent
{\bf Proof of Lemma \ref{lem:weakdeloc}.}
We use the notation introduced in the proof of Lemma
\ref{lem:unifdeloc}. Here we fix $p\ge \kappa \log^3 n/n$. Fix $\ell \in \mathbb{N}$ and write
\begin{equation}
\label{decomposition}
\left\|v_{p}\right\|_{\infty} \le \left\|\left(\frac{A_p}{\lambda_{p}}\right)^{\ell}\frac{\ol{1}}{\sqrt{n}} - v_{p}\right\|_{\infty} + \left\|\left(\frac{A_p}{\lambda_{p}}\right)^{\ell}\frac{\ol{1}}{\sqrt{n}}\right\|_{\infty}~.
\end{equation}
We bound both terms on the right-hand side. 
We start with the second term and rewrite it as 
\[
\left\|\left(\frac{A_p}{\lambda_{p}}\right)^{\ell}\frac{\ol{1}}{\sqrt{n}}\right\|_{\infty}
= \frac{1}{\sqrt{n}} \left|\frac{(n -
    1)p}{\lambda_{p}}\right|^{\ell}\left\|\left(\frac{A_p}{(n -
      1)p}\right)^{\ell}\ol{1}\right\|_{\infty}~. 
\]
Denote by $D_i = \sum\limits_{j = 1}^{n}A_{i, j}$ the degree of vertex
$i$. By standard tail bounds for the binomial distribution we have, for a fixed $i$ and $0 \le \Delta \le 1$,
\[
\PROB \left\{D_{i} < p(n - 1) - p(n - 1)\Delta\right\} \le \exp\left(\frac{-\Delta^2p(n - 1)}{2}\right)
\]
and
\[
\PROB\left\{D_{i} > p(n - 1) + p(n - 1)\Delta\right\} 
\le 
\exp\left(-\frac{3\Delta^2p(n - 1)}{8}\right)~.
\]
Using the union bound, we have
\[
\PROB\left\{ \max_i |D_{i} - p(n - 1)| > p(n - 1)\Delta\right\} \le 2(n - 1)\exp\left(-\frac{3\Delta^2p(n - 1)}{8}\right)~.
\]
We denote the event 
\[
\max_i |D_{i} - p(n - 1)| \le p(n - 1)\Delta
\]
by $E_1$. Observe that when $E_1$ holds we have $D_{i} \le p(n - 1)(1 + \Delta)$ and $D_{i} \ge p(n - 1)(1 - \Delta)$ for all $i$.
 
Assume that  $u\in \R^n$ is such that 
\begin{equation}
\label{boundforu}
\|u - \ol{1}\|_{\infty} \le 2t\Delta
\end{equation} for some $t \le \ell$. In what follows we choose
$\ell = \left\lfloor \frac{21\log n}{\log(np)}\right\rfloor$ and $\Delta =
\frac{\log(np)}{42\log n}$. Observe that $\ell\Delta \le \frac{1}{2}$.
Since $t\Delta^2 \le \ell\Delta^2 \le \frac{1}{2}\Delta$, we have
$\Delta + 2t\Delta^2 \le 2\Delta$. Thus, on the event $E_1$, using the last inequality together with \eqref{boundforu},
\begin{equation}
\label{ineq1}
\left(\frac{A_p}{(n - 1)p}u\right)_i \le \frac{p(n - 1)(1 + \Delta)(1 + 2t\Delta)}{(n -1)p} = 1 + \Delta + 2t\Delta + 2t\Delta^2 \le 1 + 2(t + 1)\Delta~.
\end{equation}
Now consider the term $\left|\frac{(n - 1)p}{\lambda_{p}}\right|^{\ell}$. Using \eqref{lambdalow} we have, with probability at least $1 - 2\exp\left(-np/2\right)$ (denote the corresponding event by $E_2$),
\[
\left|\frac{(n - 1)p}{\lambda_{p}}\right|^{\ell} \le
\left(1 - \frac{1}{\sqrt{p(n - 1)}}\right)^{-\ell}.
\]
Since $\ell \le \sqrt{p(n - 1)}$, we obtain $\left|\frac{(n - 1)p}{\lambda_{p}}\right|^{\ell} \le e$. Thus, applying \eqref{ineq1} $\ell$ times for vectors satisfying \eqref{boundforu}, on the event $E_1\cap E_2$, we have, for all $i$,
\[
\left(\left(\frac{A_p}{\lambda_p}\right)^{\ell}\ol{1}\right)_i = \left|\frac{(n - 1)p}{\lambda_{p}}\right|^{\ell}\left(\left(\frac{A_p}{(n - 1)p}\right)^{\ell}\ol{1}\right)_i \le  e( 1 + 2\ell\Delta) \le 2e~.
\]
We may similarly derive a lower bound since, for any vector satisfying \eqref{boundforu},
\begin{equation}
\label{lowerbound}
\left(\frac{A_p}{(n - 1)p}u\right)_i \ge \frac{p(n - 1)(1 - \Delta)(1 - 2t\Delta)}{(n - 1)p} = 1 - \Delta - 2t\Delta + 2t\Delta^2 \ge 1 - 2(t + 1)\Delta~.
\end{equation}
Analogously, applying \eqref{lowerbound} $\ell$ times, on the event $E_1\cap E_2$, we have
\[
\left(\left(\frac{A_p}{\lambda_p}\right)^{\ell}\ol{1}\right)_i = \left|\frac{(n - 1)p}{\lambda_{p}}\right|^{\ell}\left(\left(\frac{A_p}{(n - 1)p}\right)^{\ell}\ol{1}\right)_i \ge  \left|\frac{(n - 1)p}{\lambda_{p}}\right|^{\ell}( 1 - 2\ell\Delta) \ge 0~.
\]
Hence, on the event $E_1\cap E_2$,
\begin{equation}
\label{res1}
\left\|\left(\frac{A_p}{\lambda_{p}}\right)^{\ell}\frac{\ol{1}}{\sqrt{n}}\right\|_{\infty} \le \frac{2e}{\sqrt{n}}~.
\end{equation}
Next we bound the first term on the right-hand side of \eqref{decomposition}. Recall that for the decomposition $\ol{1}/\sqrt{n} = \alpha v_p + \beta v^{\perp}_p$ from (\ref{weakres}) we have $\alpha \ge 1 -
\frac{421}{\sqrt{np}}$ on an event $E_3$ of probability at least $1 - 4(n - 1)\exp(-np/64)$.
As before, we may write $v^{\perp}_p = \sum\limits_{i = 2}^{n}\gamma_{i}v^i_p$, where $v^i_p$ is the $i$-th orthonormal eigenvector of $A_p$.
Using $\ol{1}/\sqrt{n} = \alpha v_p + \beta v^{\perp}_p$, we have
\[
\left(\frac{A_p}{\lambda_{p}}\right)^{\ell}\frac{\ol{1}}{\sqrt{n}} = \alpha v_p + \beta\sum\limits_{i = 2}^{n}\gamma_{i}v_p^i\left(\frac{\lambda_{i}}{\lambda_{p}}\right)^{\ell},
\]
where $\lambda_{i}$ is $i$-th eigenvalue of $A_p$. 
Using F\"uredi and Koml\'os \cite[Lemmas 1 and 2]{FuKo81} once again,
for all $t \in \mathbb{R}$ we have that $|\lambda_{i}| \le \left\|A_{p} -
t\frac{\ol{1}}{\sqrt{n}}\frac{\ol{1}^T}{\sqrt{n}}\right\|$ for $i \ge 2$. 
Choosing $t = np$ we obtain $|\lambda_{i}| \le \|A_{p} - \E A_p\| +
p\|I_{n}\| \le 420\sqrt{np} + p \le 422\sqrt{np}$ on an event $E_4$ of probability at least $1 - 4(n - 1)\exp(-np/64)$.
Thus, on $E_4$ we have $\frac{|\lambda_{i}|}{\lambda_p} \le \frac{835}{\sqrt{np}}$ for $i \ge 2$, and therefore 
\begin{equation}
\label{res2}
\left\|\left(\frac{A_p}{\lambda_{p}}\right)^{\ell}\frac{\ol{1}}{\sqrt{n}} - v_{p}\right\|_{\infty} \le (1 - \alpha)\|v_p\|_{\infty} + \beta\max\limits_{i \ge 2}\left(\frac{|\lambda_{i}|}{\lambda_{p}}\right)^{\ell}~.
\end{equation}
Define $\kappa_1 = \frac{\log(835)}{\log(2\times 835^2)}$. Observe that $\kappa_1 < \frac{1}{2}$. Using $np \ge 2\times 835^2 = \kappa$,
\begin{align*}
\beta\max\limits_{i \ge 2}\left(\frac{|\lambda_{i}|}{\lambda_{p}}\right)^{\ell} &\le \beta\left(\frac{835}{\sqrt{np}}\right)^{\ell}
\\
&\le \left(\frac{835}{(np)^{\kappa_1}}\right)^\ell\exp\left((\frac{1}{2} - \kappa_1)\log(\frac{1}{np})\frac{21\log n}{\log(np)}\right) 
\\
&\le \exp\left(-21(\frac{1}{2} - \kappa_1)\log n\right) \le \frac{1}{\sqrt{n}}\ ,
\end{align*}
where we used $\left(\frac{835}{(np)^{\kappa_1}}\right)^\ell \le 1$ and the inequality $21(\frac{1}{2} - \kappa_1) > \frac{1}{2}$. Finally, on the event
$E_1\cap E_2\cap E_3\cap E_4$ we have, using the decomposition \eqref{decomposition} combined with \eqref{res1} and \eqref{res2}, that
\[
\|v_p\|_{\infty} \le \frac{1}{\alpha}\left(\frac{1 + 2e}{\sqrt{n}}\right) \le \frac{1}{1 - \frac{421}{\sqrt{np}}}\left(\frac{1 + 2e}{\sqrt{n}}\right) \le \frac{11}{\sqrt{n}}~.
\]

\subsection{Proof of Proposition \ref{prop:sqrtloglogn}}

It suffices to prove that
\[
\E\sup\limits_{p \in [0, \frac{64\log n}{n}]}|\lambda_{p} - \E\lambda_{p}| \le  5\sqrt{16 + 2\log\log n}~.
\]
Observe that
\[
\E\sup\limits_{p \in [0, 1]}|\lambda_{p} - \E\lambda_{p}| \le \E\sup\limits_{p \in [0, \frac{64\log n}{n}]}|\lambda_{p} - \E\lambda_{p}| + \E\sup\limits_{p \in [\frac{64\log n}{n}, 1]}|\lambda_{p} - \E\lambda_{p}|
\]
Let $p_0, p_1, \ldots, p_M$ be such that $0 = p_0 \le p_1 \le \cdots
\le p_M = \frac{64\log n}{n}$ and $\E(\lambda_{p_j} - \lambda_{p_{j -
    1}}) =  \eps$ for some $\eps > 0$ to be specified later. 
Such a choice is possible since $\lambda_p$ is nondecreasing in $p$. We have 
\begin{equation}
\label{eqM}
\eps M = \E\lambda_{p_M} \le \E\|A_{p_M} - \E A_{p_M}\| + \|\E A_{p_M}\| \le 170\sqrt{np_M} + np_M \le 1424\log n~.
\end{equation}
Denote for $p \in [0, p_M]$ the value $\pi_+[p] = \min\{q \in \{p_0, p_1, \ldots, p_M\}|\ q \ge p \}$ and $\pi_-[p] = \max\{q \in \{p_0, p_1, \ldots, p_M\}|\ p \geq q\}$. We have
\begin{eqnarray*}
\E\sup\limits_{p \in [0, \frac{64\log n}{n}]}|\lambda_{p} - \E\lambda_{p}|
&=&
\E\sup\limits_{p \in [0, \frac{64\log n}{n}]}\max(\lambda_{p} - \E\lambda_{p}, \E\lambda_{p} - \lambda_{p})
\\
&\le & \E\sup\limits_{p \in [0, \frac{64\log n}{n}]}\max(\lambda_{\pi_+[p]} - \E\lambda_{\pi_+[p]} + \eps, \E\lambda_{\pi_-[p]} - \lambda_{\pi_-[p]} + \eps)
\\
&= & \eps + \E\sup\limits_{p \in [0, \frac{64\log n}{n}]}\max(\lambda_{\pi_+[p]} - \E\lambda_{\pi_+[p]}, \E\lambda_{\pi_-[p]} - \lambda_{\pi_-[p]})
\\
&\le & \eps + \E\sup\limits_{q \in \{p_0, \ldots, p_M\}}|\lambda_{q} - \E\lambda_{q}|~.
\end{eqnarray*}
Since for each $p_i$, the random variable $|\lambda_{q} -
\E\lambda_{q}|$ has sub-Gaussian tails by (\ref{eq:alkrvu}), for their
maximum we obtain the bound
\[
\E\sup\limits_{q \in \{p_0, \ldots, p_M\}}|\lambda_{q} - \E\lambda_{q}| \le 4\sqrt{2\log 2M}~.
\]
Finally, using \eqref{eqM}
\[
\E\sup\limits_{p \in [0, \frac{64\log n}{n}]}|\lambda_{p} - \E\lambda_{p}| \le \inf\limits_{\eps > 0}(\eps + 4\sqrt{2\log (2848\log n/\eps)}) \le 5\sqrt{2\log (2848\log n)}~,
\]
as desired.

\subsection{Proof of Proposition \ref{prop:verysp}} 

The  proof is based on two standard facts that may be found in
  \cite{Bollobas01}.  For $k\ge 2$, let $T_k$ denote the number of
  components in a random graph $G(n, p)$ that are trees on $k$
  vertices. By Cayley's formula, $\E T_k \le \binom{n}{k}k^{k - 2}p^{k - 1}$. 
  Now we estimate the probability
  that there are trees of size at least $k + 1$. Although
  the asymptotical behaviour of this quantity is well understood, in
  what follows we need a non-asymptotic upper
  bound. By Markov's inequality and standard estimates, this probability
  is bounded by
\[
\PROB\left\{\sum\limits_{k +1}^{\infty}T_k \ge 1\right) \le \sum\limits_{j = k +1}^{\infty}\binom{n}{j}j^{j - 2}p^{j - 1} \le  \sum\limits_{j = k +1}^{\infty}\left(\frac{en}{j}\right)^{j}j^{j - 2}p^{j - 1} = \sum\limits_{j = k +1}^{\infty}\frac{en}{j^2}(enp)^{j - 1}.
\]
At the same time, Theorem 5.7 (i) in \cite{Bollobas01} states that if
$p = \frac{c}{n}$ for some $c \in [0, 1)$ then probability that
$G(n, p)$ is not a forest is bounded by
$\sum\limits_{k = 3}^{\infty}c^k = \frac{c^3}{1 - c}$. Finally, by the
estimate \eqref{eq:alkrvu}, Lemma \ref{basic}, and the monotonicty of
$\lambda_p$, with probability at least $1 - \frac{2}{n}$ we have
$ \lambda_{n^{-k/(k - 1)}} \le (173 + \sqrt{32})\sqrt{\log n} <
179\sqrt{\log n}.  $

Let $E_1$ denote the event that there are no trees of size greater
than $k + 1$, let $E_2$ denote the event that the graph is a forest,
and let $E_3$ denote the event that $\lambda_{n^{-k/(k - 1)}} <
179\sqrt{\log n}$. 
Using Jensen's inequality and the monotonicity of $\lambda_p$, we have
\[
\E\sup\limits_{p \in [0, n^{-k/(k - 1)}]}\left|\lambda_{p} -
\E\lambda_{p}\right| \le 2\E\sup\limits_{p \in [0, n^{-k/(k -
  1)}]}\left|\lambda_{p}\right| = 2\E\lambda_{n^{-k/(k - 1)}}~.
\]
Since the largest eigenvalue of 
a forest consisting of trees of size at most $k$ is bounded by $\sqrt{k-1}$,
we have, by the estimates above,
\begin{align*}
\E\lambda_{n^{-k/(k - 1)}} &\le \E\lambda_{n^{-k/(k - 1)}}\ind_{E_1 \cap E_2} + \E\lambda_{n^{-k/(k - 1)}}\ind_{E_3}(\ind_{\overline{E_1}} + \ind_{\overline{E_2}}) + 2n\PROB\{\overline{E_3}\}
\\
&\le \E\lambda_{n^{-k/(k - 1)}}\ind_{E_1 \cap E_2} + 179\sqrt{\log n}\ (\PROB\{\overline{E_1}\} + \PROB\{\overline{E_2}\})  + 2n\PROB\{\overline{E_3}\}
\\
&\le \sqrt{k-1} +  179\sum\limits_{j = k +1}^{\infty}\frac{en\sqrt{\log n}}{j^2}\left(\frac{e}{n^{1/(k - 1)}}\right)^{j - 1} + \frac{179\sqrt{\log n}}{(1 - n^{-1/(k - 1)})n^{3/(k - 1)}} + 4
\\
&\le \sqrt{k-1} + 4 + \frac{179e^{k + 1}}{(k + 1)^2}\left(\frac{\sqrt{\log n}}{n^{1/(k - 1)}}\right)/\left(1 - \frac{e}{n^{1/(k - 1)}}\right) + \frac{179\sqrt{\log n}}{(1 - n^{-1/(k - 1)})n^{3/(k - 1)}}~.
\end{align*}
The claim follows by observing that for $k \ge 2$,  $\frac{\sqrt{\log n}}{n^{3/(k - 1)}} \le \frac{\sqrt{\log n}}{n^{1/(k - 1)}} \le c_k$, where $c_k$ depends only on $k$.

\bibliographystyle{plain}

\end{document}